\documentclass[10pt]{amsart} 
\usepackage{fancyhdr}
\usepackage{enumitem}
\usepackage{amsmath,amsthm,amssymb,mathrsfs}
\usepackage{tikz-cd}
\usepackage{marginnote}
\newcommand{\filename}{Effective-Local-Height-20-Oct-2025.tex} 

\renewcommand{\geq}{\geqslant}
\renewcommand{\leq}{\leqslant}
\newcommand{\Osh}{{\mathcal O}}                        

\newcommand{\spec}{\operatorname{Spec}}

\renewcommand{\emptyset}{\varnothing}
\newcommand{\KK}{\mathbf{K}}
\newcommand{\FF}{\mathbf{F}}
\renewcommand{\AA}{\mathbb{A}} 
\newcommand{\PP}{\mathbb{P}} 
\newcommand{\QQ}{\mathbb{Q}} 
\newcommand{\RR}{\mathbb{R}} 

\newtheorem{theorem}{Theorem}[section]
\newtheorem{lemma}[theorem]{Lemma}

\theoremstyle{definition}
\newtheorem{defn}[theorem]{Definition}

\newtheorem{remark}[theorem]{Remark}
\newtheorem{example}[theorem]{Example}

\numberwithin{equation}{section}

\title[Effective calculation of local Weil functions]{Effective calculation of local Weil functions via presentations of Cartier divisors}

\author{Nathan Grieve}

\address{Laboratory for Birational Geometry, Astronomy Mathematics Building 5F, No. 1, Sec. 4, Roosevelt Rd., Taipei 10617, Taiwan (R.O.C.);
School of Mathematics and Statistics, Carleton University, 4302 Herzberg Laboratories, 1125 Colonel By Drive, Ottawa, ON, K1S 5B6, Canada; 
D\'{e}partement de math\'{e}matiques, Universit\'{e} du Qu\'{e}bec \`a Montr\'{e}al, Local PK-5151, 201 Avenue du Pr\'{e}sident-Kennedy, Montr\'{e}al, QC, H2X 3Y7, Canada;
Department of Pure Mathematics, University of Waterloo, 200 University Avenue West, Waterloo, ON, N2L 3G1, Canada
}
\email{nathan.m.grieve@gmail.com}%

\begin{document}

\begin{abstract}
We address the question of effectivity for calculation of local Weil functions from the viewpoint of presentations of Cartier divisors.  This builds on the approach of Bombieri and Gubler as well as the perspective of our earlier works.  Among other features, our approach here gives rise to theoretical effective algorithms for calculating local Weil functions on projective varieties.
\end{abstract}

\thanks{
The author thanks the Natural Sciences and Engineering Research Council of Canada for their support through his grants DGECR-2021-00218 and RGPIN-2021-03821. \\
Date: \today.  \\
File name: \filename}

\maketitle

\section{Introduction}\label{intro}

Working over a base number field $\KK$, consider the case of a geometrically integral projective variety $X$ embedded in projective $n$-space $\PP^n_{\KK}$.  Recall, that the theory of local Weil functions on $X$ play a key role in understanding the distribution and complexity of rational points in $X$.  Indeed, as emphasized by Silverman, in \cite{Silverman:1987}, such functions give a measure of  the negative logarithmic $v$-adic distance, for $v \in M_{\KK}$ a place of $\KK$, of rational points to divisors on and above $X$.  

Our purpose here is to address the foundational question of effective calculation of local Weil functions on $X$.  In this regard, our main result is Theorem \ref{explicit:local:weil:constant}.   Among other features, this result gives a refined and robust extension to \cite[Section 2.2]{Bombieri:Gubler}.   

As one example, our approach here gives rise to theoretical effective algorithms for calculating local Weil functions on $X$.   Indeed, as indicated in the proof of Lemma \ref{difference:ample:line:bundle:lemma}, effective calculation of \emph{presentations of Cartier divisors} on $X$, in the sense that we have defined in Definition \ref{presentation:Cartier:divisor:defn}, can be achieved via the theory of graded modules over a polynomial ring.  

The more subtle point then is to realize Cartier divisors on $X$ as locally free rank $1$ modules over the homogeneous coordinate ring.   Aside from this, together with sufficient knowledge of defining equations for $X$ and solutions thereof, the calculation of local Weil functions becomes entirely effective.  

Further, by this perspective, the role of Theorem \ref{explicit:local:weil:constant} is then to indicate the manner in which the constants that arise upon different choices of such presentations become effective.  The results and viewpoint that we offer here  can be seen as a starting point to the difficult and foundational question of understanding the role in which defining equations and syzygies of projective varieties govern Diophantine arithmetic questions.  

This article is organized in the following way.  In Section \ref{abs:values:conventions}, we fix conventions and other preliminaries.  In Section \ref{Presentations:Cartier:divisors} we discuss the concept of \emph{presentation of Cartier divisors} with placing emphasis on effective aspects. (See Lemma \ref{difference:ample:line:bundle:lemma} and its proof.)  We include the proof in order to indicate effective aspects thereof.

In Section \ref{effective:nullstellensatz}, we discuss effective approaches to Hilbert's Nullstellensatz (see for instance \cite{Masser:Wustholz:1983}, \cite{Kollar:1999} or \cite{Jelonek:2005}).  This becomes relevant to the theory that we describe in Sections \ref{FF:bounded:subsets} and \ref{main:thm:section}.   

In Section \ref{FF:bounded:subsets}, we develop a concept of \emph{boundedness} for subsets of rational points and with respect to a given field of definition.  This plays an important role in Section \ref{main:thm:section} which is where we prove our main result (Theorem \ref{explicit:local:weil:constant}).

\subsection*{Acknowledgements}  The author thanks the Natural Sciences and Engineering Research Council of Canada for their support through his grants DGECR-2021-00218 and RGPIN-2021-03821.  The author thanks anonymous referees for their careful reading and helpful comments that pertain to this work.

\section{Selected conventions and other preliminaries}\label{abs:values:conventions}

Let $\KK$ be a number field with set of places by $M_{\KK}$.  We fix an algebraic closure $\overline{\KK}$ of $\KK$.  All finite extension fields $\FF / \KK$ are assumed to be contained in $\overline{\KK}$.

Our conventions for representatives $|\cdot|_v$ for $v \in M_{\KK}$ follow those of \cite[Section 2.1]{Grieve:qualitative:subspace}.  For example, if $v \in M_{\KK}$ and $v | p \in M_{\QQ}$, then the restriction of $|\cdot|_v$ to $\QQ$ is $|\cdot|_p^{[\KK_v:\QQ_p]/[\KK:\QQ]}$.  So, in particular, the product formula holds true with multiplicities equal to one. 

Further, if $\FF / \KK$  is a finite extension and $v \in M_{\KK}$ then choose $w \in M_{\FF}$ with $w | v$ and put 
$$|\cdot|_{v,\KK} = |\cdot|_{v,\FF/\KK} := |\mathrm{N}_{\FF_w/\KK_v}(\cdot)|_v^{\frac{1}{[\FF_w:\KK_v]}} \text{.}$$  
Then $|\cdot|_{v,\KK}$ is an extension of $|\cdot|_v$ to $\FF$.

In what follows, unless stated otherwise, all projective varieties are assumed to be defined over $\KK$ are also assumed to be geometrically integral.  On the other hand, we oftentimes consider Cartier divisors which are defined over some finite extension of $\KK$.

\section{Presentations of Cartier divisors}\label{Presentations:Cartier:divisors}

The concept of \emph{presentation of Cartier divisors}, from \cite{Bombieri:Gubler} and developed further in \cite{Grieve:Divisorial:Instab:Vojta} and \cite{Grieve:qualitative:subspace}, allows for a robust theory of local Weil and height functions.  Here, we build on these perspective and place emphasis on effective aspects thereof.

To begin with, we place emphasis on effective aspects in the proof of the following foundational lemma.

\begin{lemma}[Compare with {\cite[Exercise II.7.5]{Hart}}]\label{difference:ample:line:bundle:lemma}  Let $X$ be a geometrically integral projective variety defined over $\KK$.
Let $D$ be a Cartier divisor on $X$ and defined over a finite extension $\FF / \KK$ of $\KK$.  Then $X$ admits globally generated line bundles $L$ and $M$ which defined over $\FF$ and which are such that 
$$
\Osh_{X_{\FF}} (D) \simeq L \otimes M^{-1} \text{.}
$$
\end{lemma}
\begin{proof}
Fix a very ample line bundle $L$ on $X$.   By a small abuse of notation denote also by $L$ its pullback to $X_{\FF}$.   Fix a sufficiently large positive integer $m_0 \gg 0$ so that the line bundle 
$$M \simeq L^{\otimes m_0} \otimes \Osh_{X_{\FF}}(D)$$ 
is globally generated.  For example, take 
$$m_0 := \operatorname{reg}_{L}(\Osh_{X_{\FF}}(D))$$ 
for $\operatorname{reg}_{L}(\Osh_{X_{\FF}}(D))$ the Castelnouvo-Mumford regularity of $\Osh_{X_{\FF}}(D)$ with respect to $L$.

Then the line bundle
$$
L^{\otimes m_0 + 1 } \otimes \Osh_{X_{\FF}}(D) \simeq L \otimes \left( L^{\otimes m_0} \otimes \Osh_{X_{\FF}} (D) \right)
$$
is very ample and hence globally generated. 

Finally, since
$$
L^{m_0 + 1} \otimes \Osh_{X_{\FF}}(D) = L \otimes M
$$
it follows that 
$$
\Osh_{X_{\FF}}(D) = L^{\otimes - m_0} \otimes M
$$
gives a description of $\Osh_{X_{\FF}}(D)$ as the difference of the globally generated line bundles $M$ and $L^{\otimes m_0}$.
\end{proof}

We recall the concept of \emph{presentation of Cartier divisors} with coefficients in a finite extension of the base number field as defined in \cite[Section 2.2]{Grieve:qualitative:subspace}.

\begin{defn}\label{presentation:Cartier:divisor:defn} Let $X$ be a geometrically integral projective variety defined over $\KK$.
Let $D$ be a Cartier divisor on $X$ and defined over a finite extension $\FF / \KK$ of the base number field $\KK$.  Denote by $s_D$ its associated meromorphic section of the line bundle $\Osh_{X_{\FF}}(D)$.
Choose globally generated line bundles $L$ and $M$ on $X$ which are such that 
$$\Osh_{X_{\FF}}(D) \simeq L \otimes M^{-1}\text{.}$$  
Fix global generating sections $s_0,\dots,s_k$ of $L$ and $t_0,\dots,t_\ell$ of $M$.  The data 
$$
\mathcal{D} = (s_D; L, \mathbf{s}; M, \mathbf{t})
$$
where 
$$\mathbf{s} = (s_0,\dots,s_k) \text{ and } \mathbf{t} = (t_0,\dots,t_\ell)$$ 
is called a \emph{presentation} of $D$ (and defined over $\FF$).
\end{defn}

\begin{remark}
Note that a first step to effectively construct presentations of $D$ is to obtain a collection of globally generating sections for $L$ and $M$.
\end{remark}


\begin{remark}
The existence of presentations, in the sense that is defined in Definition \ref{presentation:Cartier:divisor:defn}, is assured by Lemma \ref{difference:ample:line:bundle:lemma}.
\end{remark}

As in \cite{Grieve:qualitative:subspace}, and building on the theory from \cite[Section 2.2]{Bombieri:Gubler}, we consider local Weil functions with respect to presentations of Cartier divisors.

\begin{defn}\label{defn:local:Weil:function}
In the situation of Definition \ref{presentation:Cartier:divisor:defn},  fix a place $v \in M_{\KK}$ and denote by $|\cdot|_v = |\cdot|_{v,\KK}$ is associated normalized absolute value.   For $x \in X \setminus \operatorname{supp}(D)$, define
$$
\lambda_{\mathcal{D}}(x;|\cdot|_v) := \max_i \min_j \log \left| \frac{s_i}{t_j s_D}(x) \right|_v \text{.}
$$
The real valued function $\lambda_{\mathcal{D}}(x;|\cdot|_v)$ is the \emph{local Weil function} of $D$ with respect to the presentation $\mathcal{D}$ and the absolute value $|\cdot|_v$.  By construction, it has domain $X \setminus \operatorname{supp}(D)$.
\end{defn}

\begin{example} Let $X$ be a geometrically integral projective variety defined over $\KK$.
Let $s \in \FF(X)^\times$ be a nonzero rational function on $X$, defined over $\FF$, and with Cartier divisor $D := \operatorname{div}(s)$.  Then 
$$
\Osh_{X_{\FF}}(D) \simeq \Osh_{X_{\FF}}
$$
and $s$ is a meromorphic section of $\Osh_{X_{\FF}}(D)$.  Fix $v \in M_{\KK}$.  A local Weil function
$$\lambda_s = \lambda_s(\cdot; |\cdot|_v)$$ relative to $D$ is given by the presentation
$$
s = (s;\Osh_{X_{\FF}},1;\Osh_{X_{\FF}},1)
$$
that is defined by the condition that
$$
\lambda_s(x) := - \log |s(x)|_v
$$
for $x \not \in \operatorname{supp}(D) = \emptyset$.
If $s'$ is another nonzero rational function on $X$, then 
$$
\lambda_{ss'} = \lambda_{s} + \lambda_{s'}
$$
and so espeically
$$
\lambda_{s^{-1}} = - \lambda_s \text{.}
$$
\end{example}

\begin{example} Let $X$ be a geometrically integral projective variety defined over $\KK$.
Let $D_1$ and $D_2$ be two Cartier divisors on $X$, defined over a common finite extension field $\FF / \KK$ with presentations 
$$
\mathcal{D}_i = (s_{D_i};L_i,\mathbf{s}_i; M_i, \mathbf{t}_i)
$$
and with corresponding local height functions 
$$\lambda_{\mathcal{D}_i} = \lambda_{\mathcal{D}_i}(\cdot; |\cdot|_v)
\text{ for $v \in M_{\KK}$ and $i = 1,2$.}
$$

Then 
$$\mathbf{s}_1\cdot \mathbf{s}_2 := ( s_{1k} \cdot s_{2 k'})$$ 
and 
$$\mathbf{t}_1 \cdot \mathbf{t}_2 := (t_{1\ell} t_{2 \ell'})$$ 
are respective generating global sections of $L_1 \otimes L_2$ and $M_1 \otimes M_2$.  Define 
$$\lambda_{\mathcal{D}_1 + \mathcal{D}_2} (\cdot) := \lambda_{\mathcal{D}_1 + \mathcal{D}_2}(\cdot; |\cdot|_v)
$$ 
as the local height that is relative to the presentation
$$
\mathcal{D}_1 + \mathcal{D}_2 = (s_{D_1} s_{D_2}; L_1 \otimes L_2, \mathbf{s}_1 \mathbf{s}_2; M_1 \otimes M_2, \mathbf{t}_1 \mathbf{t}_2)
$$
of the divisor $D_1 + D_2$.  Then
$$
\lambda_{\mathcal{D}_1 + \mathcal{D}_2}(x) = \lambda_{\mathcal{D}_1}(x) + \lambda_{\mathcal{D}_2}(x) \text{ for $x \in X \setminus \operatorname{supp}(D_1) \bigcup \operatorname{supp}(D_2)$.}
$$
\end{example}

\begin{example}
Suppose that $D$ is a degree $d$ hypersurface in projective $n$-space $\PP^n_{\KK}$ defined by a degree $d$ homogeneous form
$$
s_D = s_D (x_0,\dots,x_n) \in \FF[x_0,\dots,x_n] 
$$
and consider the presentation
$$
\mathcal{D} = (s_D; \Osh_{\PP^n_{\FF}}(d), (x_0^d,\dots, x_0^{i_0} \cdot \hdots \cdot x_n^{i_n}, \dots, x_n^d); \Osh_{\PP^n_{\FF}}, (1)) \text{.}
$$
Then the corresponding local Weil function for $v \in M_{\KK}$ has the form
$$
\lambda_{\mathcal{D}}(x; v) := \max_{
\substack{
\mathbf{i} = (i_0,\dots,i_n) \\
|\mathbf{i}| = d
}
}
 \log \left| \frac{x_0^{i_0} \cdot \hdots \cdot x_n^{i_n}}{s_D(x)} \right|_v
\text{ 
for $x \in \PP^n$ with $s_D(x) \not = 0$.
}
$$
\end{example}

\section{Selected recall about effective Hilbert's Nullstellensatz}\label{effective:nullstellensatz}

We recall, for later use, the following effective formulation of Hilbert's Nullstellensatz (compare for example with  \cite{Masser:Wustholz:1983}, \cite{Kollar:1999} or \cite{Jelonek:2005}).  

\begin{theorem}[\cite{Jelonek:2005}]\label{effective:Hilbert:Nullstellensatz}
Let $\KK$ be a number field with algebraic closure $\overline{\KK}$.  Let $\FF / \KK$ be a finite extension in $\overline{\KK}$.  Let $U \subset \AA^n_{\KK}$ be a geometrically irreducible affine variety with defining ideal $I_U$ and affine coordinate ring
$$
\Gamma(U,\Osh_U) = \KK[U] = \KK[x_1,\dots,x_n]/I_U \text{.}
$$
  
Suppose that 
$$
f_1,\dots,f_\ell \in \FF[U_{\FF}]
$$
are non-constant and without common zeros on 
$$U_{\FF} := U \times_{\spec \KK} \spec \FF \text{.}$$ 

Then there exists effective constants 
$$N_0 = N_0(\deg(f_1),\dots,\deg(f_\ell);\dim U)$$ 
and 
$$M_0 = M_0(\deg(f_1),\dots,\deg(f_\ell);\dim U)$$ 
together with 
$$g_1,\dots,g_{\ell} \in \FF[U_{\FF}]$$ 
which are such that
$$
1 = \sum_{i=1}^\ell f_i g_i \text{,}
$$
$$
\deg(f_i g_i) \leq N_0
$$
and such that the \emph{logarithmic sizes} of each of the $g_i$ are at most equal to $M_0$.  Here, $\deg(\cdot)$ denotes the affine degree of residue classes of polynomials in $\FF[x_1,\dots,x_n]$.
\end{theorem}

\begin{proof}
See \cite{Jelonek:2005} and \cite{Masser:Wustholz:1983}.
\end{proof}

\section{Preliminaries about bounded subsets with respect to fields of definition}\label{FF:bounded:subsets}

In this section, we discuss the concept of \emph{bounded subsets} of rational points with respect to fields of definition.  This plays an important role in the proof of our main result (Theorem \ref{explicit:local:weil:constant}).

\begin{defn}[Compare with {\cite[Definition 2.2.8]{Bombieri:Gubler}}]\label{FF:bounded:subset}
Let $U$ be an geometrically integral affine variety over $\KK$.  Let $\FF / \KK$ be a finite extension field.  We say that a subset 
$$
E \subset U(\overline{\KK}) 
$$
is \emph{$\FF$-bounded}, if for all
$$
f \in \Gamma(U_{\FF},\Osh_{U_{\FF}}) 
$$
the function
$$
|f|_v \colon U_{\FF} \rightarrow \RR
$$
is bounded on $E$.
\end{defn}

Lemmas \ref{effective:Hilb:Null:lemma:1} and \ref{gauss:eqn2} below give tools for establishing $\FF$-boundedness of such subsets of rational points.
 
\begin{lemma}\label{effective:Hilb:Null:lemma:1}    Let $U$ be an affine variety over $\KK$.  Let $\FF / \KK$ be a finite extension field.  Let $f_1,\dots,f_N$ be $\FF$-algebra generators of $\Gamma(U_{\FF},\Osh_{U_{\FF}})$.   Let $E \subset U(\overline{\KK}) $.   Assume that
$$
\sup_{x \in E} \max_{j=1,\dots,N} |f_j(x)|_v < \infty \text{.}
$$
Then, $E$ is $\FF$-bounded.  In more precise terms, put
$$
\delta =
\begin{cases}
1 & \text{ if $v$ is archimedean} \\
0 & \text{ otherwise} 
\end{cases}
$$
and write
$$
f = p(f_1,\dots,f_N) 
$$
for
$$
p = p(X_1,\dots,X_{N}) = \sum_{\mathbf{j}} a_{\mathbf{j}} X^{\mathbf{j}} \in \FF[X_1,\dots,X_N] \text{.}
$$
Then, if 
$$
d := \operatorname{deg} p(X_1,\dots,X_N)
$$
is the degree of $p$, if 
$$
\operatorname{supp}(p) := \{X_1^{j_1}\cdot \hdots \cdot X_N^{j_N} : a_{j_1,\dots,j_N} \not = 0 \}
$$
is the support of $p$, and if 
$$|p|_v := \max_{\mathbf{j} = (j_1,\dots,j_N)} |a_{\mathbf{j}}|_v $$ 
is the Gauss norm of $p$, it holds true that
$$
\sup_{x \in E} |f(x)|_v \leq \left( \# \operatorname{supp} (p)\right)^{\delta} \left| p \right|_v \max \left( 1, \sup_{x \in E} \max_{j=1,\dots,N} |f_j(x)|_v \right)^{d} < \infty \text{.}
$$
\end{lemma}

\begin{proof}
Let
$$
f \in \Gamma(U_{\FF},\Osh_{U_{\FF}}) \text{.}
$$
We aim to show that the function
$$
|f|_v \colon U_{\FF} \rightarrow \RR
$$
is bounded on $E$.

To this end, writing
$$
f = p(f_1,\dots,f_N) \text{,}
$$
with
$$
p = p(X_1,\dots,X_N) \in \FF[X_1,\dots,X_N] \text{,}
$$
if
$$
d = \operatorname{deg} p
$$
and
$$
\delta =
\begin{cases}
1 & \text{ if $v$ is archimedean} \\
0 & \text{ otherwise} \text{,}
\end{cases}
$$
then
\begin{equation}\label{gauss:eqn1}
|f(x)|_v \leq (\# \operatorname{supp}(p))^\delta |p|_v \max(1,\operatorname{max}_{j = 1,\dots,N} |f_j(x)|_v)^d \text{ for each $x \in U_{\FF}$.}
\end{equation}

On the other hand, by assumption
\begin{equation}\label{gauss:eqn2}
\sup_{x \in E} \max_{j=1,\dots,N} |f_j(x)|_v < \infty \text{.}
\end{equation}
Thus, upon combining \eqref{gauss:eqn1} and \eqref{gauss:eqn2}, the conclusion is that
$$
\sup_{x \in E}|f(x)|_v \leq (\# \operatorname{supp}(p))^\delta |p|_v \max\left(1,\sup_{x \in E} \max_{j=1,\dots,N}|f_j(x)|_v\right)^d < \infty \text{.}
$$
\end{proof}

\begin{lemma}\label{effective:Hilb:Null:lemma:2}   Consider a standard finite affine covering $\{U_\ell\}_{\ell \in I}$, defined over a finite extension $\FF / \KK$, of an affine $\KK$-variety $U$, where 
$$
U_{\ell} = \{x \in U : h_\ell \not = 0 \} \subset U_{\FF}
$$ 
are standard affine open subsets, for regular functions 
$$h_\ell \in \Gamma(U_{\FF},\Osh_{U_{\FF}}) \text{.}$$  
Suppose that 
$$E \subset U(\overline{\KK})$$ 
is an $\FF$-bounded subset.  

Then there are $\FF$-bounded subsets 
$$E_\ell \subset U_\ell(\overline{\KK})\text{,}$$ 
which are such that 
$$E = \bigcup_{\ell \in I} E_\ell\text{.}$$  

In fact, for each $\ell \in I$ set
$$
E_\ell := \left\{x \in E : |h_\ell(x)|_v = \max_{k \in I} |h_k(x)|_v \right\} \text{.}
$$
Further,
 fix a collection of $\FF$-algebra generators, $f_1,\dots,f_N$, for $\Gamma(U_{\FF},\Osh_{U_{\FF}})$ and set 
$$f_{N+1} := 1 / h_\ell \text{.}$$  
Then
$$
\sup_{x \in E_{\ell}} \left| f_{N+1} (x) \right|_v \leq (\# I)^\delta \sup_{x \in E} \max_{k \in I} |g_k(x)|_v < \infty \text{.}
$$
and if 
$$
f \in \Gamma(U_{\ell},\Osh_{U_{\ell}})
$$
is written in the form
$$
f = q(f_1,\dots,f_{N+1}) \text{,}
$$
for 
$$
q = q(X_1,\dots,X_{N+1}) = \sum_{\mathbf{j}} b_{\mathbf{j}} X^{\mathbf{j}} \in \FF[X_1,\dots,X_{N+1}] \text{,}
$$
it follows that 
$$
\sup_{x \in E_{\ell}} | f(x)|_v \leq (\# \operatorname{supp}(q))^{\delta} |q|_v \max \left(1,\sup_{x \in E_{\ell}} \max_{j=1,\dots,N+1} |f_j(x)|_v \right)^{\deg q} \text{.}
$$
Here $|q|_v := \max_{\mathbf{j} = (j_1,\dots,j_N)} |b_{\mathbf{j}}|_v $ is the Gauss norm of $q$ with respect to $v$.
\end{lemma}
\begin{proof}

Let $f_1,\dots,f_N$ be a set of $\FF$-algebra generators for $\Gamma(U_{\FF},\Osh_{U_{\FF}})$.  

By assumption, there are finitely many regular functions 
$$h_\ell \in \Gamma(U_{\FF},\Osh_{U_{\FF}})\text{,}$$ 
for each $\ell \in I$, which are such that 
$$
U_{\ell} = \{x \in U : h_\ell \not = 0 \} \text{.}
$$

Since $\{U_\ell \}_{\ell \in I}$ is an affine open cover of $U_{\FF}$, the effective Hilbert's Nullstellensatz (Theorem \ref{effective:Hilbert:Nullstellensatz})  implies existence of regular functions 
$$g_\ell \in \Gamma(U_{\FF},\Osh_{U_{\FF}})$$ 
which are such that 
\begin{equation}\label{Nullstellensatz:Eqn:0}
\sum_{\ell \in I} g_\ell h_\ell = 1 \text{.}
\end{equation}
Note, that each of the $g_\ell$ and $h_\ell$, for each $\ell \in I$, can be expressed as polynomials in the $\FF$-algebra generators $f_1,\dots,f_N$.  

Also, an effective bound for the degrees and the logarithmic sizes of the $g_\ell$ can be given in terms of the $h_\ell$---recall that the $h_\ell$ can be written as polynomials in the $\FF$-algebra generators $f_1,\dots,f_N$.

As before, set
$$
\delta =
\begin{cases}
1 & \text{ if $v$ is archimedean} \\
0 & \text{ otherwise} \text{.}
\end{cases}
$$

Then, a consequence of the relation \eqref{Nullstellensatz:Eqn:0} is that 
\begin{equation}\label{Nullstellensatz:Eqn:1}
\inf_{x \in E} \max_\ell | h_\ell(x)|_v \geq \left(\# I  \right)^{-\delta} \left(\sup_{x \in E} \max_\ell |g_\ell(x)|_v \right)^{-1} > 0 \text{.}
\end{equation}

Now, for each $\ell \in I$, put
$$
E_\ell := \left\{x \in E : |h_\ell(x)|_v = \max_{k \in I} |h_k(x)|_v \right\} \text{.}
$$
Then
$$
E_{\ell} \subset U(\overline{\KK}) 
$$
and
$$
E = \bigcup_{\ell \in I} E_\ell \text{.}
$$

Recall that, by assumption, $E$ is $\FF$-bounded.  In particular,
\begin{equation}\label{Nullstellensatz:Eqn:2}
\sup_{x \in E_{\ell}} \max_{j=1,\dots,N} |f_j(x)|_v \leq \sup_{x \in E} \max_{j=1,\dots,N} |f_j(x)|_v < \infty \text{.}
\end{equation}

Fix $\ell \in I$.  Then, since $f_1,\dots,f_N,1/h_\ell$ are generators for $\Gamma(U_\ell,\Osh_{U_{\ell}})$ and because of the above Lemma \ref{effective:Hilb:Null:lemma:1} together with the inequality \eqref{Nullstellensatz:Eqn:2}, to show that $E_\ell$ is bounded, 
it suffices to show that $|1/h_\ell|_v$ is bounded on $E_\ell$.  

Now, it follows from \eqref{Nullstellensatz:Eqn:1} that
$$
\sup_{x \in E_{\ell}} \left| \frac{1}{h_\ell} \right|_v \leq (\# I)^\delta \sup_{x \in E} \max_{k \in I} |g_k(x)|_v < \infty \text{.}
$$

For the Lemma's final assertion, recall that for each fixed $\ell \in I$, 
$$
f_1,\dots,f_N,1/h_\ell
$$
are $\FF$-algebra generators of $\Gamma(U_{\ell},\Osh_{U_{\ell}})$.

Writing 
$$f_{N+1} = 1 / h_\ell\text{,}$$ 
we have shown that 
$$
\sup_{x \in E_{\ell}} \max_{j=1,\dots,N+1} |f_j(x)|_v < \infty \text{.}
$$

So, as in the previous Lemma \ref{effective:Hilb:Null:lemma:2}, if 
$$
f \in \Gamma(U_{\ell},\Osh_{U_{\ell}})
$$
is written in the form
$$
f = q(f_1,\dots,f_{N+1}) 
$$
for 
$$
q = q(X_1,\dots,X_{N+1}) = \sum_{\mathbf{j}} b_{\mathbf{j}} X^{\mathbf{j}} \in \FF[X_1,\dots,X_{N+1}] 
$$
then
$$
\sup_{x \in E_{\ell}} | f(x)|_v \leq (\# \operatorname{supp}(q))^{\delta} |q|_v \max\left(1,\sup_{x \in E_{\ell}} \max_{j=1,\dots,N+1} |f_j(x)|_v\right)^{\deg q} \text{.}
$$
\end{proof}

\section{Effective calculations of local Weil functions}\label{main:thm:section}

Let $X$ be a geometrically integral projective variety defined over $\KK$.  Let $D$ be a Cartier divisor on $X$ and defined over a finite extension $\FF / \KK$ of $\KK$.

Suppose given two \emph{presentations}
$$
\mathcal{D}_i = (s_{iD}; L_i,\mathbf{s}_i = (s_{ii,0},\dots,s_{i,n_i}); M_i, \mathbf{t}_i = (t_{i0},\dots,t_{i,m_i})) \text{ for $i = 1,2$ }
$$
of $D$ (and defined over $\FF$).  

First of all, note that the meromorphic sections $s_{1D}$ and $s_{2D}$, of the line bundle $\Osh_{X_{\FF}}(D)$, differ by a nonzero scalar multiple.  Especially
$$
s_{1D} s_{2D}^{-1} = \alpha \text{ 
for some $\alpha \in \FF^\times$.
}
$$

Consider the presentation
$$
\mathcal{D} := \mathcal{D}_1 - \mathcal{D}_2 = (s_{1D}s_{2D}^{-1}; L_1 \otimes M_2, \mathbf{s}_1 \mathbf{t}_2; M_1 \otimes L_2, \mathbf{t}_1 \mathbf{s}_2)
$$
of the trivial Cartier divisor $0$.  

It then holds true, that for each place
$
v \in M_{\KK}  
$, 
\begin{align*}
\lambda_{\mathcal{D}}(\cdot;v) & := \lambda_{\mathcal{D}_1 - \mathcal{D}_2}(\cdot;v) & \\
&= \lambda_{\mathcal{D}_1}(\cdot;v) - \lambda_{\mathcal{D}_2}(\cdot;v) \\
& = \max_{\substack{k_1 = 0,\dots,n_1 \\
\ell_2 = 0,\dots,m_2}} \min_{\substack{\ell_1 = 0,\dots,m_1 \\ k_2 = 0,\dots,n_2}} \log \left| \frac{s_{1 {k_1}} t_{2 {\ell_2}}}{  t_{1{ \ell_1}} s_{2 {k_2}} \alpha} (\cdot) \right|_v \text{.}
\end{align*}

In particular 
\begin{equation}\label{motivational:local:Weil:Function:Calculation}
\lambda_{\mathcal{D}}(\cdot;v) = \log |\alpha^{-1}|_v + \left( \max_{\substack{k_1 = 0,\dots,n_1 \\
\ell_2 = 0,\dots,m_2}} \min_{\substack{\ell_1 = 0,\dots,m_1 \\ k_2 = 0,\dots,n_2}} \log \left| \frac{s_{1 {k_1}} t_{2 {\ell_2}}  }{  t_{1 {\ell_1}} s_{2 {k_2}}  }(\cdot) \right|_v \right) \text{.}
\end{equation}

We aim to give a more explicit description of the right hand side in the above expression \eqref{motivational:local:Weil:Function:Calculation}.  This is the content of Theorem \ref{explicit:local:weil:constant}.  It gives an improvement to \cite[Theorem 2.2.11]{Bombieri:Gubler}.

\begin{theorem}\label{explicit:local:weil:constant}
Let $X$ be a geometrically integral projective variety defined over $\KK$.
Let $D$ be a Cartier divisor on $X$ and defined over a finite extension $\FF / \KK$ of $\KK$.
  Fix a place $v \in M_{\KK}$ and consider two presentations
  $$
\mathcal{D}_i = (s_{iD}; L_i,\mathbf{s}_i = (s_{ii,0},\dots,s_{i,n_i}); M_i, \mathbf{t}_i = (t_{i0},\dots,t_{i,m_i})) \text{ for $i = 1,2$ }
$$
of a Cartier divisor $D$ on $X$.  Then, there exists an effectively computable constant 
$$\mathrm{O}(1) < \infty$$ 
which is such that
$$
|\lambda_{\mathcal{D}_1}(\cdot;v) - \lambda_{\mathcal{D}_2}(\cdot;v)| \leq \mathrm{O}(1) \text{.}
$$
\end{theorem}

\begin{proof}

The goal is to establish existence of an effective constant 
$$
\mathrm{O}(1) < \infty 
$$
which is such that 
$$
- \mathrm{O}(1) \leq \lambda_{\mathcal{D}_1}(\cdot;v) - \lambda_{\mathcal{D}_2}(\cdot;v) \leq \mathrm{O}(1) \text{.}
$$

Recall that
$$
\lambda_{\mathcal{D}}(\cdot;v) = \lambda_{\mathcal{D}_1 - \mathcal{D}_2}(\cdot;v) 
$$
where
$$
\mathcal{D} := \mathcal{D}_1 - \mathcal{D}_2 = (s_{1D}s_{2D}^{-1}; L_1 \otimes M_2, \mathbf{s}_1 \mathbf{t}_2; M_1 \otimes L_2, \mathbf{t}_1 \mathbf{s}_2) \text{.}
$$
Thus, in light of \eqref{motivational:local:Weil:Function:Calculation}, 
our goal amounts to establishing existence of an effective constant 
$$\mathrm{O}(1) < \infty$$ 
which is such that
\begin{equation}\label{desired:Eqn:1}
- \mathrm{O}(1) \leq \max_{\substack{k_1 = 0,\dots,n_1 \\
\ell_2 = 0,\dots,m_2}} \min_{\substack{\ell_1 = 0,\dots,m_1 \\ k_2 = 0,\dots,n_2}} \log \left| \frac{s_{1 {k_1}} t_{2 {\ell_2}}  }{  t_{1 {\ell_1}} s_{2 {k_2}}  }(\cdot) \right|_v \leq \mathrm{O}(1) \text{.}
\end{equation}

In establishing \eqref{desired:Eqn:1}, by interchanging the role of $\mathcal{D}_1$ and $\mathcal{D}_2$, it suffices to establish the right most inequality---namely to prove existence of a constant $\mathrm{O}(1) < \infty$ which is such that 
\begin{equation}\label{desired:Eqn:2}
\max_k \min_{\ell} \log \left| \frac{s_k}{t_\ell}(\cdot) \right|_v \leq \mathrm{O}(1) \text{.}
\end{equation} 
Here, in \eqref{desired:Eqn:2}, to reduce subscript notation we put 
$$s_k := s_{1 {k_1}} t_{2 {\ell_2}} \text{ for $k = 0,\dots, n_1m_2$}$$ 
and 
$$t_\ell := t_{1 {\ell_1}} s_{2 {k_2}} \text{ for $\ell = 0,\dots, m_1 n_2$.}$$

To establish the upper bound \eqref{desired:Eqn:2}, choose a closed embedding of $X$ into some projective space $\PP^N$
$$
X \hookrightarrow \PP^N_{\KK} = \PP^N_{[x_0:\dots:x_N]} \text{.}
$$
For each $i=0,\dots,N$, let 
$$U_i \subset X_{\FF}$$ 
be the affine open subset
$$
U_i := \left\{ x = [x_0:\dots : x_N] \in X_{\FF} : x_i \not = 0 \right\} \text{.}
$$
Then on the affine open subsets
$$
U_{i\ell} := \{x \in U_i : t_{\ell}(x) \not = 0 \} 
$$
the functions
$$
g_{k \ell} := s_k / t_{\ell}
$$
are regular.

On the other hand, the functions
$$
f_{ij} := \frac{x_j}{x_i} \text{,}
$$
for $j = 0,\dots,N$, and each fixed $ i = 0, \dots, N$, generate $\Gamma(U_i,\Osh_{U_i})$ as an $\FF$-algebra.

So, define the sets $E_i$, for each $i = 0,\dots,N$, by the condition that
$$
E_i := \left\{ x = [x_0:\dots:x_N] \in X(\overline{\KK}) : |x_i|_v = \max_j |x_j|_v \right\} \text{.}
$$
In particular, by construction, note that if $x \in E_i$, then
$$
\max_{j=0,\dots,N} \left| f_{ij}(x) \right|_v = 1 \text{.}
$$

It follows that each of the sets $E_i$ are $\FF$-bounded in $U_i$, for $i = 0,\dots,N$.  Thus, there exists a standard affine covering $\{U_{i \ell}\}$, of each of the sets $U_i$, and $\FF$-bounded subsets
$$
E_{i \ell} \subset U_{i\ell}
$$
which are such that 
$$
E_i = \bigcup_{i} E_{i\ell}
$$
and 
$$
\sup_{x \in E_{i \ell}} \max_k \left| g_{k \ell}(x) \right|_v < \infty \text{.}
$$

Thus, since the sets $E_{i \ell}$ cover $X(\overline{\KK})$, the existence of the desired finite and effectively computable constant $\gamma < \infty$ then follows as in Lemma \ref{effective:Hilb:Null:lemma:2}.

In more explicit terms, fix $\FF$-algebra generators
$$
f_{i0} := \frac{x_0}{x_i} , \dots, \hat{f}_{ii},\dots,f_{iN}:= \frac{x_N}{x_i}, f_{iN+1} := \frac{1}{h_{i\ell}}
$$
for each $\FF$-algebra $\Gamma(U_{i\ell},\Osh_{U_{i\ell}})$.  Here
$$
h_{i\ell} \in \Gamma(U_\ell,\Osh_{U_\ell})
$$
is a suitable regular function and the notation 
$
\hat{f}_{ii} 
$
means omit
$$
f_{ii} := \frac{x_i}{x_i} = 1
$$
in the chosen generating set.

Then, if 
$$
g_{k\ell} \in \Gamma(U_{i\ell},\Osh_{U_{i\ell}})
$$
is written in the form
$$
g_{k\ell} = q_{ik\ell}(f_{i0},\dots,\hat{f_{ii}},\dots,f_{iN},f_{iN+1}) \text{,}
$$
for
$$
q_{ik\ell} = q_{ik\ell}(X_1,\dots,X_{N+1}) \in \FF[X_1,\dots,X_{N+1}] \text{,}
$$
then 
\begin{multline*}
\sup_{x \in E_{i \ell}} \left| g_{k\ell}(x) \right|_v \leq \\ 
\left( \# \operatorname{supp}(q_{ik\ell}) \right)^{\delta} \left|q_{ik\ell} \right|_v \max\left(1,\sup_{x \in E_{\ell}} \max_{j=0,\dots, \hat{i}, \dots,N+1} |f_{ij}(x)|_v\right)^{\deg q_{ik\ell}} \text{.}
\end{multline*}

The above discussion then shows that the constant $\mathrm{O}(1)$ that arises in \eqref{desired:Eqn:2} may be taken to satisfy the inequality
\begin{multline*}
\mathrm{O}(1) \leq \\
\max_{k,\ell,i} \log^+ \left( \left( \# \operatorname{supp}(q_{ik\ell}) \right)^\delta |q_{ik\ell}|_v \max\left(1,\sup_{x \in E_{\ell}} \max_{j=0,\dots, \hat{i}, \dots,N+1} |f_{ij}(x)|_v\right)^{\deg q_{ik\ell}} \right) \text{.}
\end{multline*}
The existence and effectivity of the constant $\mathrm{O}(1)$ that arises in the conclusion of Theorem \ref{explicit:local:weil:constant} then follows too because of \eqref{desired:Eqn:1} and \eqref{motivational:local:Weil:Function:Calculation}.
\end{proof}

\providecommand{\bysame}{\leavevmode\hbox to3em{\hrulefill}\thinspace}
\providecommand{\MR}{\relax\ifhmode\unskip\space\fi MR }
\providecommand{\MRhref}[2]{%
  \href{http://www.ams.org/mathscinet-getitem?mr=#1}{#2}
}
\providecommand{\href}[2]{#2}

\end{document}